\documentclass[a4paper]{article}
\usepackage[utf8]{inputenc}
\usepackage[T1]{fontenc} 
\usepackage{fancyhdr}
\usepackage[colorlinks,linkcolor= black,citecolor=blue]{hyperref}

\usepackage[style = alphabetic, sorting = nyt, useprefix]{biblatex}
\usepackage[title]{appendix}
\usepackage{amsfonts,amsmath,amssymb,amsthm}
\usepackage{enumerate}
\usepackage{bbm} 
\usepackage{verbatim}
\usepackage{authblk}
\usepackage{times}
\usepackage{mdframed}
\usepackage{lipsum}
\usepackage{soul}

\usepackage{graphicx,float}
\usepackage{tikz,tkz-euclide,pgfplots}
\usetikzlibrary{calc,patterns}
\usetikzlibrary{arrows,shapes,positioning}
\usepackage{epigraph} 
\usepackage{csquotes}
\SetBlockThreshold{2}

\DeclareMathOperator{\Red}{Red}
\DeclareMathOperator{\Int}{Int}

\DeclareMathOperator{\Cay}{Cay}
\DeclareMathOperator{\Cancel}{Cancel}
\DeclareMathOperator{\degg}{deg}
\newcommand{\moinsun}{^{-1}}

\newcommand\flr[1]{\lfloor #1 \rfloor}
\newcommand{\vide}{\textup{\O}}

\newcommand\SetCond[2]{\left\{ #1 \;\middle|\; #2 \right\}}
\newcommand\Pres[2]{\langle #1 | #2 \rangle}
\def\tilde{\widetilde}

\newtheorem*{thm*}{Theorem}
\newtheorem*{cor*}{Corollary}

\newtheorem{thm}{Theorem}[section]
\newtheorem{defi}[thm]{Definition}
\newtheorem{lem}[thm]{Lemma}
\newtheorem{defilem}[thm]{Definition-Lemma}

\theoremstyle{remark}
\newtheorem{rem}[thm]{Remark}

\title{Parallel geodesics and minimal stable length of random groups}
\author{\textsc{Tsung-Hsuan Tsai}}
\affil{Institut Camille Jordan, Université Lyon 1\\
\small{tsai@math.univ-lyon1.fr}}
\date{}

\pgfplotsset{compat=1.18}
\begin{document}
\maketitle

\begin{abstract}
We show that for any pair of long enough parallel geodesics in a random group $G_\ell(m,d)$ with $m$ generators at density $d<1/6$, there is a van Kampen diagram having only one layer of faces. Using this result, we give an upper bound, depending only on $d$, of the number of pairwise parallel geodesics in $G_\ell(m,d)$ when $d<1/6$. As an application, we show that the minimal stable length of $G_\ell(m,d)$ at $d<1/6$ is exactly $1$.
\end{abstract}
\tableofcontents

\newpage

\section{Introduction}

\paragraph{Random groups.} A \textit{random group} $G_\ell(m,d)$ with $m\geq 2$ generators at density $d\in[0,1]$, as defined by Gromov in \cite{Gro93}, is given by the presentation \[G = \Pres{x_1,\dots,x_m}{R_\ell},\] where $R_\ell$ is a set of $(2m-1)^{d\ell}$ relators, randomly chosen from the set of cyclically reduced words on $\{x_1^\pm,\dots,x_m^\pm\}$ of length at most $\ell$.

We say that $G_\ell(m,d)$ satisfies some group property \textit{asymptotically almost surely} (denoted a.a.s.) if the probability that the property holds converges to $1$ as the relator length $\ell$ goes to infinity. The first result in the theory of random groups is a phase transition (\cite{Gro93}): at $d>1/2$, a.a.s. $G_\ell(m,d)$ is a trivial group; whereas at $d<1/2$, a.a.s. $G_\ell(m,d)$ is a non-elementary hyperbolic group.

\paragraph{Parallel geodesics.} Two geodesics in (the Cayley graph of) a $\delta$-hyperbolic group are called \textit{parallel} if they do not intersect and their end points are respectively $10\delta$-close. By convention, if the mentioned geodesics are bi-infinite, they share the same boundary points.

In this article, we study van Kampen diagrams between parallel geodesics, referred to as \textit{band diagrams} (Definition-Lemma \ref{band diagram}). A \textit{1-layer face} in a band diagram is a face that touches both geodesics. We first discover that 1-layer faces appear in band diagrams of random groups at density $d<1/4$.

\begin{thm*}[Theorem \ref{parallel diagram in one forth}]
    If $d<1/4$, then a.a.s. for any pair of parallel geodesics $\gamma_1$, $\gamma_2$ of $G_\ell(m,d)$ of lengths at least $\frac{1200}{1-4d}\ell$ and any band diagram $D$ between them, there is a face of $D$ that intersects the two boundary geodesics.
\end{thm*}

We then show that at $d<1/6$, which corresponds to $C'(1/3)$ in random groups, band diagrams already have only one layer of faces.

\begin{thm*}[Theorem \ref{parallel diagram in one sixth}]
    Let $L\geq10000\ell$. If $d<1/6$, then a.a.s. for any pair of parallel geodesics $\gamma_1, \gamma_2$ of $G_\ell(m,d)$ of lengths at least $L$, there exists a pair of parallel sub-geodesics $\gamma_1',\gamma_2'$ of lengths at least $L-7200\ell$ having a 1-layer band diagram.
\end{thm*}

To prove this result, we introduce a new small cancellation condition, called the $\tilde C(p)$ condition (Definition \ref{def C tilde of p}), which is a strong version of the $C(p)$ condition. We then established phase transitions for the $\tilde C(p)$ conditions in random groups (Lemma \ref{C tilde of 2} and Lemma \ref{C tilde of p}).

\paragraph{Canonical representatives.} The notion of \textit{canonical representatives} was introduced by Rips and Sela in \cite{RS95} to solve equations over hyperbolic groups. They then asked whether every Gromov hyperbolic group admits globally stabled canonical representatives. A recent work by Lazarovich and Sageev \cite{LS24} shows that hyperbolic cubulated groups admit such representatives, and it is known from Olivier and Wise \cite{OW11} that random groups at density $d<1/6$ are cubulable, and hence admit globally stabled canonical representatives.

In Appendix of \cite{RS95}, Rips and Sela remarked that van Kampen diagrams between geodesics of a $C'(1/8)$ group are 1-layer diagrams. Using this structure, they constructed globally stabled canonical representatives for $C'(1/8)$ groups. Together with their construction, our Theorem \ref{parallel diagram in one sixth} provides an alternative method for showing that random groups at density $d<1/6$ have globally stabled canonical representatives.

\paragraph{Number of parallel geodesics.} It is known that the number of pairwise parallel bi-infinite geodesics in a hyperbolic group is bounded above by a large number (\cite{Gro87}, \cite[Proposition 3.1]{Del96}). In \cite{GM18}, Gruber and Mackay showed that for a random \textit{triangular} group at density $d\in]1/3,d_{crit}[$, there exists a large $n_0 = n_0(d)$ such that, for any $n\geq n_0$, the $n$-periodic Burnside quotient is infinite. A key ingredient in their proof is providing a much sharper bound on the number of parallel geodesics.

In this article, we provide a sharp upper bound on the number of parallel geodesics in $G_\ell(m,d)$ at density $d<1/6$:
\begin{thm*}[Theorem \ref{number of parallel geodesics}] Let $k\geq 1$ be an integer. Let $G_\ell(m,d)$ be a random group at density $d<1/6$. There exists a number $C=C(d)$ such that, if $\gamma_1,\dots,\gamma_k$ is a set of $k$ parallel geodesics of length at least $C\ell$, then a.a.s.
\[k\leq 2+\frac{4d}{1-6d}.\]
\end{thm*}

\quad

Let $R, L, N >0$. A hyperbolic group $G=\Pres{X}{R}$ (acting on itself) is $(R,L,N)$-acylindrical if, for every $x,y\in \Cay(X,R)$ with $|x-y|\geq L$, the number of elements $g\in G$ satisfying simultaneously $|gx-x|\leq R$ and $|gy-y|\leq R$ is at most $N$ (\cite[Definition 5.5]{Cou18}). Following Gruber and Mackay's acylindricity result for random triangular groups \cite[Proof of Theorem 1.5 by Proposition 1.3]{GM18}, we have a version for random groups at density $d<1/6$.

\begin{cor*} If $d<1/6$, then for any $a>0$, a.a.s. $G_\ell(m,d)$ is \[(a\ell, C\ell+a\ell, k(C\ell+2a\ell)^2)\]-acylidrical where $k$ and $C$ are from Theorem \ref{number of parallel geodesics}.
\end{cor*}

Despite this result, it is not known whether there exists a uniform $n_0 = n_0(m,d)$ such that for $n\geq n_0$, the $n$-periodic Burnside quotient of $G_\ell(m,d)$ with $d<1/6$ is infinite. The main difficulty lies in the fact that the hyperbolicity constant of $G_\ell(m,d)$ is linear in $\ell$ (see \cite{Gro93}, \cite{Oll07} and Theorem \ref{inequality for diagrams}), whereas the hyperbolicity constant of a random triangular group depends solely on its density.

\paragraph{Minimal stable length.} The \textit{minimal stable length} of a hyperbolic group $G$ is the infimum of the stable lengths $[u]^\infty$ among the loxodromic elements $u\in G$. According to \cite[Proposition 3.1]{Del96}, if $G$ is a hyperbolic group with at most $k$ parallel geodesics, then the stable length $[u]^\infty$ of a loxodromic element $u\in G$ is at least $[u]/k$ where $[u]$ is the translation length of $u$. 

As an application of Theorem \ref{number of parallel geodesics}, we show that, in a random group $G_\ell(m,d)$ at density $d<1/6$, a.a.s. if $[u] < \frac{1-3d-3\varepsilon}{2k}\ell$ then $[u]^\infty = [u]$ (Remark \ref{rem stable length}). This leads to the following result:

\begin{thm*}[Theorem \ref{inj rad at 1/6}] If $d<1/6$, then a.a.s. the minimal stable length of $G_\ell(m,d)$ is $1$.
\end{thm*}

\textbf{Acknowledgements.} I would like thank my thesis advisor, Thomas Delzant, for inspiring the idea of parallel geodesics in random groups during my time in Strasbourg. I am also very thankful to Rémi Coulon for inviting me to Dijon and for the valuable discussions we had on this subject.

\section{Preliminary: van Kampen diagrams and random groups}
\paragraph{Van Kampen diagrams.}
We consider edge-oriented combinatorial 2-complexes as described in \cite[\S III.2]{LS77}. A vertex is also called a point or a $0$-cell; a face is called a $2$-cell. Every edge $e$ has a starting point, an ending point and an inverse edge $e\moinsun$. A pair of inverse edges $\{e,e\moinsun\}$ is called an undirected edge or a $1$-cell. For a 2-complex $Y$, denote by $Y^{(i)}$ the set of $i$-cells of $Y$ (for $i=0,1,2$) and $|Y| = |Y^{(2)}|$ the number of faces of $Y$. For simplicity, if $e$ is an (oriented) edge of $Y$, we also write $e \in Y^{(1)}$; similarly, if $f$ is a face of $Y$, we write $f \in Y$.

A \textit{path} is a sequence of edges $p = e_1\dots e_n$ such that the ending point of $e_i$ is the starting point of $e_{i+1}$ for each $i$. A path is called \textit{reduced} if it has no subpath of type $ee\moinsun$. A \textit{loop} is a path that ends at its starting point. A loop is \textit{cyclically reduced} if it is a reduced loop with the additional requirement that $e_n\neq e_1\moinsun$. A path (or a loop) is called \textit{simple} if it does not cross itself.

In this article, we introduce an additional structure on 2-complexes for convenience: every face $f$ is assigned a starting point and an orientation on its boundary, referred to as the \textit{boundary loop} and denoted by $\partial f$.\\

Let $G = \Pres{X}{R}$ be a finite group presentation where every relator $r\in R$ is cyclically reduced. The \textit{Cayley complex} $\Cay(X,R)$ of $\Pres{X}{R}$ is the 2-complex constructed from its Cayley graph by attaching a face on each relator loop. Every edge of $\Cay(X,R)$ is labeled by the generator it represents, and every face is naturally labeled by the relator read on its boundary loop.

\begin{defi}
    A \textup{diagram} with respect to the group presentation $\Pres{X}{R}$ is a finite connected 2-complex $Y$ with labels on edges by generators and labels on faces by relators such that:
    \begin{itemize}
        \item If an edge $e$ is labeled by $x\in X$, then its inverse $e\moinsun$ is labeled by $x\moinsun$.
        \item If a face $f$ is labeled by $r\in R$, then its boundary loop is labeled by the word $r$.
    \end{itemize}
\end{defi}

Since the relators are cyclically reduced, every boundary loop of is a reduced loop. For every diagram $Y$, we associate a 2-complex morphism $\varphi_Y : Y\to \Cay(X,R)$ that maps $i$-cells to $i$-cells (for $i = 0,1,2$) and preserves the labels on edges and faces. The map $\varphi_Y$ is unique up to label-preserving isometries on $\Cay(X,R)$.\\

A pair of faces in a diagram is called \textit{reducible} if they are labeled by the same relator and their boundary loops share a common edge at the same respective position with the same orientation (Figure \ref{reducible faces}; the two faces share their fourth edge with the same orientation). A diagram is called \textit{reduced} if it contains no reducible pair of faces. Equivalently, a diagram is reduced if the map $\varphi_Y$ is locally injective on $Y\setminus Y^{(0)}$.

\begin{figure}[h]
    \centering
    \begin{tikzpicture}
        \fill[gray!20] (0,0) -- (0,1) --+ (30:1) -- (30:2);
        \fill[gray!20] (0,0) -- (-30:1) --+ (30:1) -- (30:2);
        \fill[gray!20] (0,0) -- (0,1) --+ (150:1) -- (150:2);
        \fill[gray!20] (0,0) -- (210:1) --+ (150:1) -- (150:2);
        \draw[very thick] (0,0) -- (0,1) --+ (30:1) -- (30:2);
        \draw[very thick] (0,0) -- (-30:1) --+ (30:1) -- (30:2);
        \draw[very thick] (0,1) --+ (150:1) -- (150:2);
        \draw[very thick] (0,0) -- (210:1) --+ (150:1) -- (150:2);
        \draw [thick, ->] (30:2) arc (0:-30:1.5);
        \draw [thick, ->] (150:2) arc (180:210:1.5);
        \fill (30:2) circle (0.07);
        \fill (150:2) circle (0.07);
        \node at (30:1) {$r$};
        \node at (150:1) {$r$};
    \end{tikzpicture}
    \caption{A reducible pair of faces.}
    \label{reducible faces}
\end{figure}
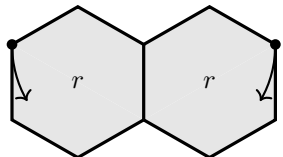

The \textit{reduction degree} of $Y$ is the total number of edges causing reducible pairs of faces, counted with \textit{multiplicity}: for any edge $e\in Y^{(1)}$, any relator $r\in R$, and any integer $j$, we count the number of faces $f\in Y$ labeled by $r$ that have $e$ as the $j$-th boundary edge. If this number is $t$, we add $(t-1)^+$ to the reduction degree, where $(\cdot)^+$ denotes the positive part function. Here is a formal definition.

\begin{defi}[Reduction degree, {\cite[Definition 2.5]{GM18}}]\label{reduction degree}
    Let $Y$  be a diagram of a group presentation $G = \langle X|R\rangle$. Denote by $\varphi_2: Y^{(2)}\to R$ the labeling function on faces. Let $\ell$ be the maximal boundary length of faces of $Y$. 

    The \textup{reduction degree} of $Y$ is defined as
    \[\Red(Y) := \sum_{e\in Y^{(1)}}\sum_{r\in R}\sum_{1\leq j\leq \ell}\Big(\big|\{f\in Y \,|\, \varphi_2(f) = r, e \textup{ is the $j$-th edge of } \partial f\}\big| - 1 \Big)^+.\]
\end{defi}

Note that $Y$ is reduced if and only if $\Red(Y) = 0$.\\

A \textit{van Kampen diagram} $D$ is a simply connected diagram that is embedded in the Euclidean plane. Up to cyclic permutations and inversion, the boundary loop of $D$ (which may not be reduced) is denoted by $\partial D$. A \textit{disk diagram} is a van Kampen diagram that is homeomorphic to a disk. In this case, the boundary loop is simple. Here are two useful results on van Kampen diagrams.

\begin{thm}[Van Kampen's lemma \cite{vK33}]
    Let $G = \Pres{X}{R}$ be a finitely presented group. A word $w$ of $X^\pm$ is trivial in $G$ if and only if there exists a reduced van Kampen diagram $D$ such that a chosen boundary loop $\partial D$ is labeled by the word $w$.
\end{thm}

\begin{thm}[\cite{Gro87}]
    A finitely presented group $G = \Pres{X}{R}$ is Gromov hyperbolic if and only if there exists a number $C>0$ such that every reduced van Kampen diagram of $G$ satisfies the isoperimetric inequality
    \[|\partial D| \leq C|D|.\]
\end{thm}

\paragraph{Random groups.} Fix an alphabet $X = \{x_1,\dots,x_m\}$ with $m\geq 2$. For $\ell\geq1$, let $B_\ell$ be the set of cyclically reduced words on $X^\pm$ of length at most $\ell$. A \textit{sequence of random groups} $(G_\ell(m,d))_{\ell\in\mathbb{N}}$ with $m$ generators at density $d\in \{-\infty\}\cup[0,1]$ is defined by 

\[G_\ell(m,d) = \Pres{X}{R_\ell}\] 
where $R_\ell$ is a random subset of $B_\ell$, uniformly distributed among all subsets of $B_\ell$ of cardinality $\flr{|B_\ell|^d}$.

Let $(Q_\ell)$ be a sequence of group presentation properties. We say that $G_\ell(m,d)$ satisfies $Q_\ell$ \textit{asymptotically almost surely} (denoted a.a.s.) if the probability that $G_\ell(m,d)$ satisfies $Q_\ell$ converges to $1$ when $\ell$ goes to infinity.\\

Recall that (\cite{Gro93}) for any $m\geq 2$, if $d>1/2$ then a.a.s. $G_\ell(m,d)$ is a trivial group; if $d<1/2$ then a.a.s. $G_\ell(m,d)$ is a hyperbolic group. More precisely:

\begin{thm}[{\cite{Gro93}}, {\cite{Oll07}}]\label{inequality for diagrams} If $d<1/2$, then for any $\varepsilon>0$, a.a.s. every reduced van Kampen diagram $D$ of $G_\ell(m,d)$ satisfies the isoperimetric inequality
\[|\partial D|\geq (1-2d-\varepsilon)|D|\ell.\]
In addition, a.a.s. $G_\ell(m,d)$ is torsion free and $\delta$-hyperbolic with 
    \[\delta = \frac{4\ell}{1-2d}.\]
\end{thm}

\quad

The isoperimetric inequality for reduced van Kampen diagrams in random groups was first proved with the additional condition that the diagrams have a bounded number of faces (i.e., $|D|<K$ with some arbitrary $K>0$ before a.a.s.). A version of the inequality for non-reduced 2-complex diagrams was established in \cite{Tsa22b} (see also \cite{GM18}), with the additional condition that the 2-complexes have "bounded complexity."

A vertex in a 2-complex is called \textit{non-singular} if it is a degree $2$ vertex that is not the starting point of a boundary loop. Other vertices are called \textit{singular}. A \textit{maximal arc} in a 2-complex is a reduced path that passes only through non-singular vertices and has singular vertices as its endpoints. Given any finite 2-complex, its 1-skeleton can be partitioned into maximal arcs. The \textit{complexity} of a 2-complex encodes the number of maximal arcs along with the number of faces:

\begin{defi}[Complexity of a 2-complex]\label{complexity}
Let $Y$ be a finite connected 2-complex. Let $K>0$. We say that $Y$ has complexity $K$ if $|Y|\leq K$ and if for any face $f$ of $Y$, the boundary loop $\partial f$ is divided into at most $K$ maximal arcs.
\end{defi}

If $D$ is a planar and simply connected 2-complex with $|D|\leq K$, then the complexity of $D$ is $6K$. Here is the non-reduced 2-complex version of the isoperimetric inequality.

\begin{thm}[\cite{Tsa22b}] \label{inequality for 2-complex} Let $d<1/2$. For any (small) $\varepsilon>0$ and any (large) $K>0$, a.a.s. every diagram $Y$ of complexity $K$ of $G_\ell(m,d)$ satisfies the inequality 
\[|Y^{(1)}|+\Red(Y)\geq (1-d-\varepsilon)|Y|\ell.\]
\end{thm}

\quad

The \textit{cancellation} of a 2-complex $Y $ is the number of cancelled edges when attaching faces. It is defined by
\[\Cancel(Y):=\sum_{e\in Y^{(1)}}(\degg(e)-1)^+\]
where $\degg(e)$ is the number of faces attached on the edge $e$. For any 2-complex $Y$ with face boundary length at most $\ell$, we have
\[|Y^{(1)}|+\Cancel(Y) = \sum_{f\in Y}|\partial f|\leq |Y|\ell.\]
So Theorem \ref{inequality for 2-complex} can be rephrased as

\begin{equation}\label{useful reduced inequality}
    \Cancel(Y)-\Red(Y)\leq (d+\varepsilon)|Y|\ell.
\end{equation}

If $D$ is a reduced van Kampen diagram, we have simply $\Cancel(D) \leq (d+\varepsilon)|D|\ell.$ Together with the isoperimetric inequality (Theorem \ref{inequality for diagrams}), we get

\begin{equation}\label{useful diagram inequality}
    \Cancel(D)\leq \frac{d+\varepsilon}{1-2d-\varepsilon}|\partial D|.
\end{equation}

\section{Small cancellations in random groups}
Throughout this article, in the figures containing diagrams, we will use light gray regions to indicate faces and dark gray regions to indicate subdiagrams. The following lemma will be useful to prove phase transition results in random groups.
\begin{lem}[{\cite[Corollary 1.7]{Tsa22b}}]\label{existence of diagrams} Let $\varepsilon>0$ and $K>0$. Let $(D_\ell)$ be a sequence of finite planar simply connected 2-complexes of the same geometric form (\cite{Tsa22b} Definition 3.1) such that every face boundary length of $D_\ell$ is exactly $\ell$. If every sub-2-complex $D_\ell'$ of $D_\ell$ satisfies
    \[|\partial D'_\ell|\geq (1-2d+\varepsilon)|D'_\ell|\ell,\]
then a.a.s. there exists a reduced van Kampen diagram of $G_\ell(m,d)$ whose underlying 2-complex is $D_\ell$.
\end{lem}\quad

Recall that a group presentation satisfies the $C(p)$ small cancellation condition if any face of a reduced van Kampen diagram is surrounded by at least $p$ faces \cite{LS77}. For random groups, there is a phase transition at density $1/p$ (see \cite{OW11}, \cite{Tsa22b}).

\begin{lem}[{\cite[Proposition 4.2]{Tsa22b}}]\label{C of p} Let $p\geq 2$ be an integer. 
\begin{enumerate}[(i)]
    \item If $d<1/p$, then a.a.s. $G_\ell(m,d)$ satisfies $C(p)$.
    \item If $d>1/p$, then a.a.s. $G_\ell(m,d)$ does not satisfy $C(p)$.
\end{enumerate}
\end{lem}

\quad

We introduce here a stronger condition, denoted $\tilde C(p)$.

\begin{defi}\label{def C tilde of p} Let $p\geq 2$ be an integer. A group presentation satisfies the $\tilde C(p)$ small cancellation condition if, for every reduced van Kampen diagram, any simply connected subdiagram is surrounded by at least $p$ faces.
\end{defi}

It is clear that $\tilde C(p)$ implies $C(p)$ for any $p\geq2$. The condition $\tilde C(2)$ is equivalent to the property that every face in a reduced van Kampen diagram has a simple boundary loop. In \cite[Ch.V, Lemma 4.1]{LS77}, it was shown that this property holds for $C(6)-T(3)$ groups, $C(4)-T(4)$ groups, and $C(3)-T(6)$ groups. In random groups, there is a phase transition at density $2/5$.

\begin{lem}\label{C tilde of 2}Let $G_\ell(m,d)$ be a random group at density $d$. 
    \begin{enumerate}[(i)]
        \item If $d<2/5$, then a.a.s. $G_\ell(m,d)$ satisfies $\tilde C(2)$.
        \item If $d>2/5$, then a.a.s. $G_\ell(m,d)$ does not satisfy $\tilde C(2)$.
    \end{enumerate}
\end{lem}
\begin{proof}\quad
    \begin{enumerate}[(i)]
        \item Notice that $d<2/5<1/2$ so the group is $C(2)$. We refer to the proof of Lemma \ref{C tilde of p} (i) with $p=2$ and without the inequality $2/(p+3)\geq 1/p$ (which is the reason that we need to separate the case $p=2$).
        \item By Lemma \ref{existence of diagrams} with some $\varepsilon<d-2/5$, there exists a reduced diagram of the following form (Figure \ref{not C tilde of 2}), which denies the $\tilde C(2)$ condition.
        
        \begin{figure}[h]
        \begin{center}
        \begin{tikzpicture}
            \fill [gray!20] circle (1.5);
            \draw [thick] circle (1.5);
            \draw [thick] (-0.5,0) circle (1);
            \draw [thick] (-1.5,0) -- (0.5,0);
            \node at (-0.5,0.2) {\scriptsize{$0.8\ell$}};
            \node at (-0.5,0.8) {\scriptsize{$0.2\ell$}};
            \node at (-0.5,-0.8) {\scriptsize{$0.2\ell$}};
            \node at (1.9,0) {\scriptsize{$0.6\ell$}};
        \end{tikzpicture}
        \end{center}
            \caption{A van Kampen diagram that is not $\tilde C(2)$.}
            \label{not C tilde of 2}
        \end{figure}
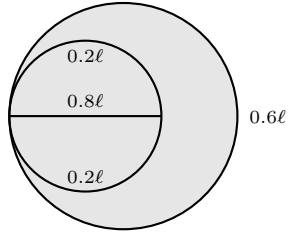
    \end{enumerate}
\end{proof}

There is also a phase transition of $\tilde C(p)$, at density $1/p$, for any $p\geq 3$.

\begin{lem}\label{C tilde of p}
    Let $G_\ell(m,d)$ be a random group at density $d$. Let $p\geq 3$ be an integer.
    \begin{enumerate}[(i)]
        \item If $d<1/p$, then a.a.s. $G_\ell(m,d)$ satisfies $\tilde C(p)$.
        \item If $d>1/p$, then a.a.s. $G_\ell(m,d)$ does not satisfy $\tilde C(p)$.
    \end{enumerate}
\end{lem}
\begin{proof}\quad
\begin{enumerate}[(i)]
    \item Assume that at density $d<1/p$, a.a.s. there were a reduced diagram $D$ with a simply connected subdiagram $D'$ that is surrounded by $p-1$ faces $f_1,\dots,f_{p-1}$.
    
    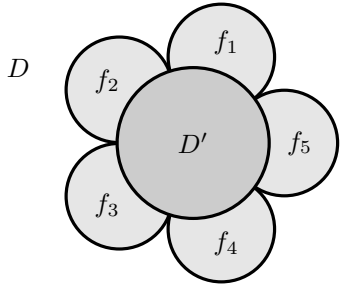
\begin{figure}[h]
    \begin{center}
    \begin{tikzpicture}
    \foreach \x in {1,...,5}{
        \fill[gray!20] (72*\x:1.2) circle (0.7);
        \draw[very thick] (72*\x:1.2) circle (0.7);
        \node at (72*\x:1.4) {$f_\x$};
        \begin{scope}[shift={(72*\x:1.2)}]
        \end{scope}
    }
    \fill[gray!40] (0,0) circle (1);
    \draw[very thick] (0,0) circle (1);
    \node at (0,0) {$D'$};
    \node at (-2.3,1) {$D$};
    \end{tikzpicture} 
    \end{center}
        \caption{A van Kampen diagram that is not $\tilde C(6)$}
        \label{fig small cancellation tilde p}
    \end{figure}
    
    Apply the isoperimetric inequality (Theorem \ref{inequality for diagrams}) on $D'$ and $D$ with some $\varepsilon>0$ to be determined, we have
    \[|\partial D'| \geq (1-2d-\varepsilon)|D'|\ell,\]
    \[|\partial D| \geq (1-2d-\varepsilon)(|D'|+p-1)\ell.\]
    On the other hand, \[|\partial D| + |\partial D'| \leq \sum_{i=1}^{p-1}|\partial f_i|\leq (p-1)\ell.\]
    So $(p-1)\ell\geq (1-2d-\varepsilon)(2|D'|+p-1)\ell$, or
    \[d\geq \frac{|D'|}{2|D'|+p-1}-\varepsilon/2.\]

    At $d<1/p$ the group satisfies $C(p)$ (Lemma\ref{C of p} (i)), so $|D'|\geq 2$. Together with $p\geq 3$, we have $\frac{|D'|}{2|D'|+p-1}\geq \frac{2}{3+p} \geq \frac{1}{p}$. Now we choose $\varepsilon = 1/p-d>0$, which leads to a contradiction.

    \item By Lemma \ref{C of p} (ii) and  the fact that $\tilde C(p)$ implies $C(p)$.
\end{enumerate}
\end{proof}

\begin{lem}\quad\label{connected on boundary}
    If $d<1/4$, then a.a.s. any reduced van Kampen diagram $D$ of $G_\ell(m,d)$ satisfies the following property:
    
    If $\gamma$ is a sub-path of the boundary loop $\partial D$ whose image is a geodesic in the group, then for any face $f$ of $D$, the intersection $\partial f\cap \gamma$ is connected.
\end{lem}
\begin{proof}
    Suppose that there were a face $f$ of $D$ such that $\partial f \cap \gamma$ has at least two connected components. Let $v_1$ be one of the components and $v_2$ a component next to it. Let $p$ be the segment on $\gamma$ between $v_1$ and $v_2$. The face $f$ together with the segment $p$ bounds a disk subdiagram $D'$ of $D$, in which the path $v_1pv_2$ is a geodesic. In $D'$, let $q$ be the interior segment of $\partial f$ and $r$ be the exterior segment of $\partial f$ excluding $v_1$ and $v_2$ (Figure \ref{figure not connected boundary}). The path $r$ may intersect $\gamma$ but it does not matter.
\begin{figure}[h]
\begin{center}
\begin{tikzpicture}
\begin{scope}
    \clip (-1.55,0) rectangle (1.55,1.55);
    \fill [gray!20] circle (1.5);
    \fill [gray!40] (0.1,0) circle (0.7);
    \draw [thick] (0,0) circle (1.5);
    \draw [thick] (0.1,0) circle (0.7);
\end{scope}
\draw [thick] (-2,0) -- (2,0);
\node at (-1.3,1.3) {$D'$};
\node at (0,1) {$f$};
\node at (-1,-0.2) {\scriptsize{$v_1$}};
\node at (1.2,-0.2) {\scriptsize{$v_2$}};
\node at (0.7,0.6) {\scriptsize{$q$}};
\node at (1.2,1.2) {\scriptsize{$r$}};
\node at (0.1,-0.2) {\scriptsize{$p$}};
\node at (2.2,0) {\scriptsize{$\gamma$}};
\end{tikzpicture}
\end{center}
    \caption{A face having two connected components on a geodesic.}
    \label{figure not connected boundary}
\end{figure}
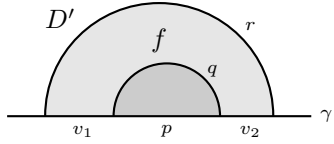

    Note that $q$ does not intersect $\gamma$, so the two paths $p$ and $q$ bounds a subdiagram, and $|D'|\geq 2$. Since $\gamma$ is a geodesic, we have $|p|\leq |q|$, so \[|\partial D'| = |r|+|v_1|+|p|+|v_2| \leq |r|+|v_1|+|q|+|v_2| =  |\partial f|\leq \ell.\] Apply the isoperimetric inequality on $D'$ with $d<1/4$ and some small $\varepsilon$, we get
    \[|\partial D'| \geq (1-2d-\varepsilon)|D'|\ell > \frac{1}{2}|D'|\ell.\]
    The above two inequalities give $|D'|<2$, which is impossible.
\end{proof}

\section{Diagrams between parallel geodesics (Band diagrams)}

Let $x$ and $y$ be vertices in a 2-complex $Y$. A \textit{geodesic} from $x$ to $y$, denote by $[x, y]\subset Y$, is a path $\gamma:[0,t]\to Y$ such that $\gamma(0)=x$, $\gamma(t)=y$ and no other path can be shorter. Denote by $|x - y|$ the distance between $x$ and $y$, which is the length of any geodesic path joining $x$ and $y$. 

A \textit{bi-infinite geodesic} is a bi-infinite path such that every sub-path is a geodesic. If $x$ and $y$ are already on some given geodesic $\gamma$, then $[x, y]$ refers to the sub-geodesic of $\gamma$ joining $x$ and $y$, unless otherwise specified.

\begin{defi} Let $G = \Pres{X}{R}$ be a $\delta$-hyperbolic group. Two geodesics $\gamma_1:[0,t_1]\to\Cay(X,R)$ and $\gamma_2:[0,t_2]\to\Cay(X,R)$ are called \textit{parallel} if they do not intersect and their endpoints are respectively $10\delta$-close. That is, $|\gamma_1(0)-\gamma_2(0)|\leq 10\delta$ and $|\gamma_1(t_1)-\gamma_2(t_2)|\leq 10\delta$.
\end{defi} 

By convention, two bi-infinite geodesics are parallel if they do not intersect and they have the same pair of boundary points on $\partial\Cay(X,R)$ (refer \cite{CDP90}).\\

\begin{defilem}[Band diagrams]\label{band diagram} Let $G = \langle X|R \rangle$ be a $\delta$-hyperbolic space. Let $\gamma_1:[0,t_1]\to\Cay(X,R)$ and $\gamma_2:[0,t_2]\to\Cay(X,R)$ be two parallel geodesics of lengths at least $20\delta$ in $G$. There exists a \textup{reduced disk diagram} $D$ together with an associated map $\varphi:D\to \Cay(G,X)$ such that:
\begin{itemize}
    \item There are four distinguished vertices $x_1,y_1,y_2,x_2$ on $\partial D$ that divide $\partial D$ into four paths $\overline{x_1y_1}$, $\overline{y_1y_2}$, $\overline{y_2x_2}$ and $\overline{x_2x_1}$,
    
    \item $\varphi(\overline{x_iy_i})$ is a sub-geodesic of $\gamma_i$ for $i = 1,2$,

    \item $\varphi(\overline{x_1x_2})$ is a sub-geodesic of some $[\gamma_1(0),\gamma_2(0)]$ and $\overline{y_1y_2}$ is a sub-geodesic of some $[\gamma_1(t_1),\gamma_2(t_2)]$.
\end{itemize}
The diagram $D$ (together with the map $\varphi$) constructed above is called a \textup{band diagram} between the two geodesics $\gamma_1$ and $\gamma_2$. The vertices $x_1,x_2,y_1,y_2$ on $\partial D$ are called the \textup{corners} of $D$.
\end{defilem}

\begin{figure}[h]
\begin{center}
\begin{tikzpicture}
\fill [gray!40] (-3,0) -- (3,0) -- (3,2) -- (-3,2);
\node at (0,1) {$\varphi(D)$};

\draw [very thick] (-3.5,0) -- (3.5,0);
\draw [very thick] (-3.5,2) -- (3.5,2);
\node at (3.8,0) {$\gamma_2$};
\node at (3.8,2) {$\gamma_1$};

\draw [thick] (-3,0) -- (-3,2);
\draw [thick] (3,0) -- (3,2);

\fill (-3,0) circle (0.05);
\fill (3,0) circle (0.05);
\fill (-3,2) circle (0.05);
\fill (3,2) circle (0.05);
\fill (-3.5,2) circle (0.05);
\fill (3.5,2) circle (0.05);
\fill (-3.5,0) circle (0.05);
\fill (3.5,0) circle (0.05);
\node at (-3,2.2) {\scriptsize{$\varphi(x_1)$}};
\node at (3,2.2) {\scriptsize{$\varphi(y_1)$}};
\node at (-3,-0.2) {\scriptsize{$\varphi(x_2)$}};
\node at (3,-0.2) {\scriptsize{$\varphi(y_2)$}};

\end{tikzpicture}
\end{center}
    \caption{A band diagram.}
    \label{diagram between geodesics}
\end{figure}
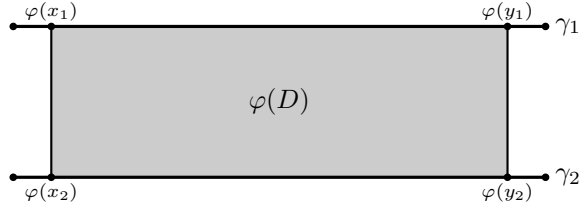

\begin{proof} 
    Choose geodesics $[\gamma_1(0), \gamma_2(0)]$ and $[\gamma_1(t_1), \gamma_2(t_2)]$ such that they intersect $\gamma_1$ and $\gamma_2$, respectively, with only one connected component.
    
    Let $[x_1',x_2']$ be the sub-geodesic of $[\gamma_1(0), \gamma_2(0)]$ that intersect $\gamma_1$ and $\gamma_2$ respectively on one single vertex. Similarly, obtain $[y_1',y_2']\subset [\gamma_1(t_1), \gamma_2(t_2)]$. Since the lengths of $\gamma_1$ and $\gamma_2$ are at least $20\delta$, the geodesics $[x_1',x_2']$ and $[y_1',y_2']$ do not intersect, and the loop $p = [x_1',y_1'] \cup [y_1',y_2'] \cup [y_2',x_2'] \cup [x_2',x_1']$ has no proper subloop.
    
    By van Kampen's lemma, there exists a reduced van Kampen diagram $D$ bounded by the labeling word of $p$, and a natural map $\varphi:D\to \Cay(X,R)$ that maps $\partial D$ isomorphically on the loop. Since the word of $p$ has no proper subword that is trivial in $G$, $D$ is homeomorphic to a disk. The four distinguished vertices $x_1,x_2,y_1,y_2$ on $\partial D$ are the pre-images of $x_1',x_2',y_1',y_2'$.
\end{proof}

\quad

\begin{thm}\label{parallel diagram in one forth}
    If $d<1/4$, then a.a.s. for any pair of parallel geodesics $\gamma_1$, $\gamma_2$ in $G_\ell(m,d)$ of lengths at least $\frac{1200}{1-4d}\ell$ and any band diagram $D$ between them, there is a face of $D$ that intersects the two boundary geodesics.
\end{thm}

\begin{proof}
    Recall that a.a.s. $G_\ell(m,d)$ is $\delta$-hyperbolic with $\delta = \frac{4\ell}{1-2d}$ (Theorem \ref{inequality for diagrams}). At $d<1/4$, assume simply $\delta\leq 10\ell$.

    Since $\frac{1200}{1-4d}\ell\geq 200\ell\geq 20\delta$, there is a band diagram $D$ between $\gamma_1,\gamma_2$. Let $L$ be the length of $\gamma_1$. As the length of $\gamma_2$ is at most $L+20\delta$, we have \[|\partial D|\leq 2L+400\ell.\]
   
    Let $x_1,x_2,y_1,y_2$ be the corners of $D$. Let $k_i$ be the number of maximal arcs on the boundary geodesic $[x_i,y_i]$ for $i=1,2$. Suppose, by contradiction, that there were no face of $D$ that intersects $[x_1,y_1]$ and $[x_2,y_2]$ simultaneously. Given that every face can intersect a boundary geodesic at most once (Lemma \ref{connected on boundary}), every face of $D$ contributes at most one maximal arc on the two geodesics. Thus, we have $k_1+k_2\leq |D|$.
     
\begin{figure}[h]
\begin{center}
\begin{tikzpicture}
\fill [gray!40] (-3,0) -- (3,0) -- (3,2) -- (-3,2);
\node at (0,1) {$D$};

\draw [thick] (-3,0) -- (3,0);
\draw [thick] (-3,2) -- (3,2);
\draw [thick] (-3,0) -- (-3,2);
\draw [thick] (3,0) -- (3,2);

\fill (-3,0) circle (0.05);
\fill (3,0) circle (0.05);
\fill (-3,2) circle (0.05);
\fill (3,2) circle (0.05);
\node at (-3,2.2) {\scriptsize{$x_1$}};
\node at (3,2.2) {\scriptsize{$y_1$}};
\node at (-3,-0.2) {\scriptsize{$x_2$}};
\node at (3,-0.2) {\scriptsize{$y_2$}};
\node at (0,2.3) {\scriptsize{$k_1$ maximal arcs}};
\node at (0,-0.3) {\scriptsize{$k_2$ maximal arcs}};

\end{tikzpicture}
\end{center}
    \caption{A diagram with $k_i$ maximal arcs in $[x_i,y_i]$.}
\end{figure}
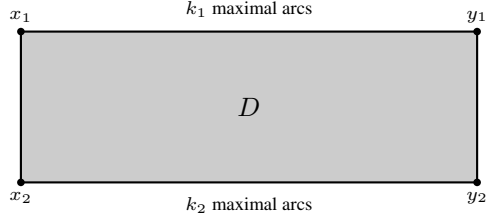
    
    On the other hand, for $i=1,2$, since $[x_i,y_i]$ is a geodesic word, any of its maximal arcs cannot be longer than $\ell/2$, which gives $k_i\times\frac{\ell}{2}\geq |[x_i,y_i]|\geq L-200\ell$. Therefore, we get \[|D|\geq k_1+k_2\geq \frac{4L}{\ell}-800.\]

    Apply the isoperimetric inequality $|\partial D| \geq (1-2d-\varepsilon)\ell|D|$ (Theorem \ref{inequality for diagrams}), we get
    \[2L+400\ell \geq (1-2d-\varepsilon)(4L-800\ell),\]
    or
    \[L\leq \frac{200(3-4d-2\varepsilon)\ell}{1-4d-2\varepsilon}.\]   
    
    With $\varepsilon = (1-4d)/4$, we obtain $L<\frac{1200}{1-4d}\ell$, which leads to a contradiction.

\end{proof}

\quad

A face in a band diagram $D$ that intersects the two boundary geodesics is called a \textit{1-layer face}. If $D$ has only 1-layer faces, it is called a \textit{1-layer band diagram}.

\begin{lem}\label{one layer diagram} Let $G_\ell(m,d)$ be a random group at density $d<1/6$. If $D$ be a band diagram of $G_\ell(m,d)$ consisting of at least two 1-layer faces, then every face of $D$ between these two faces is a 1-layer face.
\end{lem}
\begin{proof}
    
Suppose, by contradiction, that there were 1-layer faces $f_1,f_2$ of $D$ that bound a non-empty subdiagram $D_{int}$ having no 1-layer faces. Denote by $D_{ext}$ the union of $D_{int}$ with $f_1$ and $f_2$. There are three cases:

\begin{enumerate}[\textbf{Case} 1.]
    \item $f_1$ and $f_2$ intersect at the two boundary geodesics.
    
    \begin{figure}[h]
\begin{center}
\begin{tikzpicture}
\fill [gray!20] (-2,0) -- (2,0) -- (2,2) -- (-2,2);
\draw (0,0) -- (0,2);
\fill [gray!40] (0,1) circle (0.7);
\node at (-1.2,1) {$f_1$};
\node at (1.2,1) {$f_2$};
\node at (-2.5,1.5) {$D_{ext}$};
\node at (0,1) {$D_{int}$};
\draw [thick] (-2,0) -- (2,0);
\draw [thick] (-2,2) -- (2,2);
\draw [thick] (-2,0) -- (-2,2);
\draw [thick] (2,0) -- (2,2);
\draw (0,1) circle (0.7);
\end{tikzpicture}
\end{center}
    \caption{Case 1 of $D_{int}$}
    \end{figure}
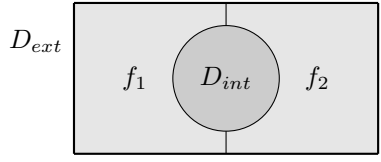

    By Lemma \ref{C tilde of p}, at density $d<1/3$ the group satisfies the $\tilde C(3)$ condition, but the subdiagram $D_{int}$ is surrounded by 2 faces. Contradiction.
    
    \item $f_1$ and $f_2$ intersect at one of the two boundary geodesics. Other intersections would lead to Case 1.
    
    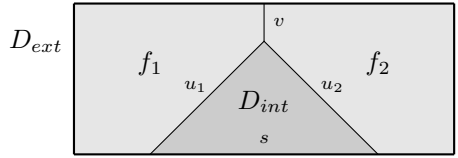
\begin{figure}[h]
\begin{center}
\begin{tikzpicture}
\fill [gray!20] (-2.5,0) -- (2.5,0) -- (2.5,2) -- (-2.5,2);
\fill [gray!40] (-1.5,0) -- (0,1.5) -- (1.5,0);
\node at (-1.5,1.2) {$f_1$};
\node at (1.5,1.2) {$f_2$};
\node at (-0.9,0.9) {\scriptsize{$u_1$}};
\node at (0.9,0.9) {\scriptsize{$u_2$}};
\node at (0.2,1.75) {\scriptsize{$v$}};
\node at (0,0.2) {\scriptsize{$s$}};
\node at (0,0.7) {$D_{int}$};
\node at (-3,1.5) {$D_{ext}$};
\draw [thick] (-2.5,0) -- (2.5,0);
\draw [thick] (-2.5,2) -- (2.5,2);
\draw [thick] (-2.5,0) -- (-2.5,2);
\draw [thick] (2.5,0) -- (2.5,2);
\draw (-1.5,0) -- (0,1.5) -- (1.5,0);
\draw (0,1.5) -- (0,2);
\end{tikzpicture}
\end{center}
    \caption{Case 2 of $D_{int}$}
    \end{figure}
    
    Let $u_i$ be the intersection path of $f_1$ and $D_{int}$ for $i=1,2$, $s$ be the remaining path on $\partial D_{int}$ and $v$ be the intersection path of $f_1,f_2$ (which may be empty). 
    
    Let us estimate the boundary length of $D_{ext}$. First, $|\partial D_{ext}| = |\partial f_1|+|\partial f_2|+|s|-|u_1|-|u_2|-2|v|$. Since $s$ is a geodesic path, it is shorter than $|v_1|+|v_2|$. We have
    \[|\partial D_{ext}|\leq |\partial f_1|+|\partial f_2|\leq 2\ell.\]
    
    Together with the isoperimetric inequality 
    \[|\partial D_{ext}|\geq (1-2d-\varepsilon)|D_{ext}|\ell,\] 
    
    we get
    \[d\geq \frac{|D_{ext}|-2}{2|D_{ext}|}-\varepsilon/2.\]

    As $|D_{ext}|\geq 3$, we have $d\geq 1/6-\varepsilon/2$, contradiction with any $\varepsilon<1/6-d$.

    \item $f_1$ and $f_2$ do not intersect at boundary geodesics. Other intersections would lead to Case 2.
    
    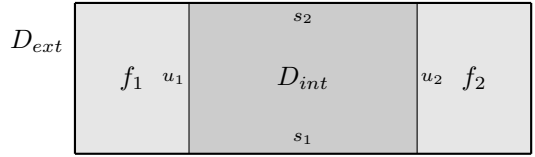
\begin{figure}[h]
\begin{center}
\begin{tikzpicture}
\fill [gray!20] (-3,0) -- (3,0) -- (3,2) -- (-3,2);
\fill [gray!40] (-1.5,0) -- (1.5,0) -- (1.5,2) -- (-1.5,2);
\node at (-2.25,1) {$f_1$};
\node at (2.25,1) {$f_2$};
\node at (-1.7,1) {\scriptsize{$u_1$}};
\node at (1.7,1) {\scriptsize{$u_2$}};
\node at (0,0.2) {\scriptsize{$s_1$}};
\node at (0,1.8) {\scriptsize{$s_2$}};
\node at (0,1) {$D_{int}$};
\node at (-3.5,1.5) {$D_{ext}$};
\draw [thick] (-3,0) -- (3,0);
\draw [thick] (-3,2) -- (3,2);
\draw [thick] (-3,0) -- (-3,2);
\draw [thick] (3,0) -- (3,2);
\draw (-1.5,0) -- (-1.5,2);
\draw (1.5,0) -- (1.5,2);
\end{tikzpicture}
\end{center}
    \caption{Case 3 of $D_{int}$}
    \end{figure}
    
    Let $u_i$ be the intersection path of $f_i$ and $D_{int}$ for $i=1,2$, and let $s_i$ be the intersection path of $\gamma_i$ and $D_{int}$ for $i=1,2$.
    
    By Lemma \ref{connected on boundary}, if a face $f$ in $D_{int}$ intersects $s_1$, then the intersection is a connected maximal arc, and we have $|\partial f\cap s_1|\leq \frac{1}{2}|\partial f|$. So
    \[|s_1| \leq \sum_{f\in D_{int}, f\cap s_1\neq \vide} |\partial f\cap s_1|\leq \frac{1}{2}\sum_{f\in D_{int}, f\cap s_1\neq \vide} |\partial f|.\]

    By the same inequality on $s_2$ and the hypothesis that $D_{int}$ has no 1-layer face, we get
    \[|s_1|+|s_2| \leq \frac{1}{2}\sum_{f\in D_{int}} |\partial f|\leq \frac{1}{2}|D_{int}|\ell.\]

    So \[|\partial D_{int}|+|\partial D_{ext}| = |\partial f_1|+|\partial f_2|+2|s_1|+2|s_2|\leq 2\ell+|D_{int}|\ell.\]
    Together with the isoperimetric inequalities on $D_{int}$ and on $D_{ext}$: 
    \[(1-2d-\varepsilon)|D_{int}|\ell \leq |\partial D_{int}|,\]
    \[(1-2d-\varepsilon)|D_{ext}|\ell \leq |\partial D_{ext}|;\]
    we get
    \[(1-2d-\varepsilon)(|D_{int}|+|D_{ext}|)\ell\leq 2\ell+|D_{int}|\ell.\]
    With the fact that $|D_{ext}|=|D_{int}|+2$, we obtain
    \[d\geq \frac{|D_{int}|}{|D_{int}|+4}-\varepsilon/2.\]
    Since $|D_{int}|\geq 2$, we have $d\geq 1/6-\varepsilon/2$, contradiction with any $\varepsilon<1/6-d$.
\end{enumerate}
\end{proof}

\begin{rem}
    The estimations in Case 2 and Case 3 hold for any density $d<1/4$. We actually have an upper bound on the number of faces $|D_{int}|$ between two 1-layer faces that depends only on the density $d$. In Case 2 we have $|D_{int}|< \frac{4d}{1-2d}$; and in Case 3 we have $|D_{int}|< \frac{4d}{1-4d}$.
\end{rem}

\begin{thm}\label{parallel diagram in one sixth}
     Let $L\geq10000\ell$. If $d<1/6$, then a.a.s. for any pair of parallel geodesics $\gamma_1, \gamma_2$ of $G_\ell(m,d)$ of lengths at least $L$, there exists a pair of parallel sub-geodesics $\gamma_1',\gamma_2'$ of lengths at least $L-7200\ell$ having a 1-layer band diagram.

     In addition, the endpoints of $\gamma_i'$ are $3600\ell$-close to the endpoints of $\gamma_i$ for $i = 1,2$; and any pair of parallel sub-geodesics $\gamma_1'', \gamma_2''$ of $\gamma_1', \gamma_2'$ has a 1-layer band diagram by restricting a 1-layer band diagram of $\gamma_1', \gamma_2'$.
\end{thm}
\begin{proof} Let $D$ be a band diagram between $\gamma_1,\gamma_2$ with associated map $\varphi : D\to\Cay(G,X)$. By Theorem \ref{parallel diagram in one forth}, there are 1-layer faces in $D$. 

Note that when $d<1/6$, we have $\frac{1200}{1-4d}\ell\leq 3600\ell$. Let $f_x$ be the 1-layer face of $D$ such that $\varphi(f_x)$ is closest to $\gamma_1(0)$. The face $\varphi(f_x)$ must be $3600\ell$-close to $\gamma_1(0)$. Otherwise, since $|\partial f_x|\leq \ell\leq 20\delta$, the connected component of $D\backslash f_x$ at the side of $\gamma_1(0)$ is a subdiagram of $D$ that can be identified, via a restriction of $\varphi$, as a band diagram between some sub-geodesics of $\gamma_1,\gamma_2$ of length at least $ \frac{1200}{1-4d}\ell$ that has no 1-layer faces, contradicting Theorem \ref{parallel diagram in one forth}. Let $f_y$ be the 1-layer face of $D$ such that $\varphi(f_y)$ is closest to $\gamma_1(t_1)$. By the same arguments $\varphi(f_y)$ is $3600\ell$-close to $\gamma_1(t_1)$. Since $L\geq 10000\ell$, the two faces $f_x$ and $f_y$ are different.

Let $D'$ be the subdiagram of $D$ consisting of $f_x,f_y$ and the faces between them. According to Lemma \ref{one layer diagram}, $D'$ is a one layer diagram. Now let $\gamma_1' = \varphi(D')\cap\gamma_1$ and $\gamma_2' = \varphi(D')\cap\gamma_2$. It is a pair of parallel geodesics of lengths at least $L-7200\ell$. We then identify $D'$, via a restriction of $\varphi$, as a one layer band diagram between $\gamma_1'$ and $\gamma_2'$.

\begin{figure}[h]
\begin{center}
\begin{tikzpicture}
\fill [gray!40] (-3.5,0) -- (3.5,0) -- (3.5,1) -- (-3.5,1);
\fill [gray!20] (-1.5,0) -- (-2.5,0) -- (-2.5,1) -- (-1.5,1);
\fill [gray!20] (1.5,0) -- (2.5,0) -- (2.5,1) -- (1.5,1);
\node at (0,0.5) {$\varphi(D')$};
\node at (-2,0.5) {\scriptsize{$\varphi(f_x)$}};
\node at (2,0.5) {\scriptsize{$\varphi(f_y)$}};

\draw [very thick] (-4,0) -- (4,0);
\draw [very thick] (-4,1) -- (4,1);
\node at (4.3,0) {$\gamma_2$};
\node at (4.3,1) {$\gamma_1$};

\draw [thick] (-3.5,0) -- (-3.5,1);
\draw [thick] (3.5,0) -- (3.5,1);
\draw (2.5,0) -- (2.5,1);
\draw (1.5,0) -- (1.5,1);
\draw (-2.5,0) -- (-2.5,1);
\draw (-1.5,0) -- (-1.5,1);

\end{tikzpicture}
\end{center}
    \caption{The 1-layer faces $f_x$ and $f_y$ in the diagram $D$.}
\end{figure}
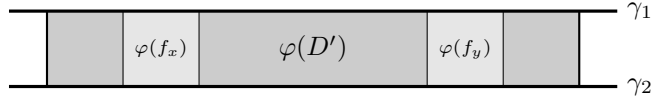

\end{proof}

\section{Number of parallel geodesics}

\begin{thm}\label{number of parallel geodesics} Let $k\geq 1$ be an integer. Let $G_\ell(m,d)$ be a random group at density $d<1/6$. There exists $C=C(d,k  )$ such that, if $\gamma_1,\dots,\gamma_k$ is a set of $k$ parallel geodesics of lengths at least $C\ell$, then a.a.s.
\[k\leq 2+\frac{4d}{1-6d}.\]
In particular, if $d<1/10$, then a.a.s. $k\leq 2$.
\end{thm}

Recall, from Theorem \ref{inequality for diagrams}, that a.a.s. $G_\ell(m,d)$ is $\delta$-hyperbolic with $\delta = \frac{4\ell}{1-2d}$. At $d<1/4$, we assume simply $\delta\leq 10\ell$. We shall construction a 2-complex diagram based on the $k$ parallel geodesics $\gamma_1,\dots,\gamma_k$ of $G_\ell(m,d)$, similar to \cite{GM18}. Every statement in this section is considered under the assumption that it holds asymptotically almost surely (a.a.s.) as $\ell\to\infty$.\\

Suppose that $C\geq 10000$. By Theorem \ref{parallel diagram in one sixth}, there are parallel sub-geodesics $\gamma_1',\dots,\gamma_k'$ of lengths at least $C\ell-7200\ell$ such that, for any $i<j$ there is a one layer band diagram $D_{ij}$ between $\gamma_i'$ and $\gamma_j'$. Denote by $\varphi_{ij}: D_{ij}\to \Cay(X,R_\ell)$ the associated map of $D_{ij}$. If $i>j$, then $D_{ij}$ refers to the diagram $D_{ji}$. Recall that
\[|\partial D_{ij}| \leq 2C\ell+200\ell.\]

By the construction of a band diagram (Definition-Lemma \ref{band diagram}), the restriction of $\varphi_{ij}$ to the boundary $\partial D_{ij}$ is injective. Let $Y$ be the diagram obtained by attaching the boundaries of the $\binom{k}{2}$ reduced disk diagrams $(D_{ij})_{1\leq i<j\leq k}$ along the $k$ geodesics via their associated maps. By the isoperimetric inequality (Theorem \ref{inequality for diagrams}), 
\[|D_{ij}|\leq \frac{|\partial D_{ij}|}{(1-2d-\varepsilon)\ell} \leq \frac{2C+200}{1-2d-\varepsilon}.\] Thus,
\begin{equation}\label{inequality number of faces}
    |Y| = \sum_{1\leq i<j\leq k}|D_{ij}|\leq \binom{k}{2}\frac{2C+200}{1-2d-\varepsilon}.
\end{equation}

With a simple estimation of the number of maximal arcs in $Y$, the complexity of $Y$ is at most \[6\binom{k}{2}\frac{2C+200}{1-2d-\varepsilon}.\]

Hence, with Theorem \ref{inequality for 2-complex}, we can apply Inequality (\ref{useful reduced inequality}):
\begin{equation*}\tag{\ref{useful reduced inequality}}
    \Cancel(Y)-\Red(Y)\leq (d+\varepsilon)|Y|\ell.
\end{equation*}

\quad

\begin{lem}\label{cancel of Y}
    \[\Cancel(Y)\geq \sum_{1\leq i<j\leq k}\Cancel(D_{ij}) + k(k-2)(C-8000)\ell.\]
\end{lem}
\begin{proof} 
    The cancellation of $Y$ is the sum of the cancellations of the $\binom{k}{2}$ sub-diagrams $(D_{ij})_{1\leq i\leq j}$ plus the cancellations on the $k$ geodesics $(\gamma_i')_{1 \leq i \leq k}$.

    Recall that the length of $\gamma_j'$ is at least $C\ell-7200\ell$. For any $i$, the corner of $D_{ij}$ on $\gamma_j'$ at the side of $\gamma_j'(0)$ is $100\ell$-close to $\gamma_j'(0)$. Thus, for any distinct $i,i',j$, there are at least $C\ell-8000\ell$ cancellations when attaching $D_{ij}$ and $D_{i'j}$. Consequently, there are at least $(k-2)(C-8000)\ell$ cancellations on the geodesic $\gamma_j$. Hence the result.
\end{proof}

\quad\\

Now we need an estimation on the reduction degree $\Red(Y)$. Since every $D_{ij}$ is reduced, the reductions in $Y$ may only happen on the $k$ geodesics.

For any face $f$ of $D_{ij}$, denote by $\Int(f)$ the set of \textit{intermediate edges} of $f$, which are the edges of $\partial f$ that do not intersect the two boundary geodesics. Note that $|\partial f| = |\partial f\cap \gamma_i| + |\partial f\cap \gamma_j| + |\Int(f)|$.

\begin{defi}
    An edge $e$ on $\partial D_{ij}\cap \gamma_i$ is said to be "close to the other side" in $D_{ij}$ if 
    \begin{enumerate}
        \item The face $f$ of $D_{ij}$ containing $e$ intersects $\gamma_j$, and
        \item $|\Int(f)|\leq |\partial f\cap \gamma_j|$.
    \end{enumerate}
    Otherwise, $e$ is said to be "far from the other side".
\end{defi}

\begin{lem} \label{lem estimation far from the other side}
    Let $E_{ij}$ be the set of edges on $\partial D_{ij}$ that are far from the other side in $D_{ij}$. We have
    \[|E_{ij}|\leq 4\Cancel(D_{ij}) + 400\ell.\]
\end{lem}
\begin{proof}
    Let $f$ be a face of $D_{ij}$. If $f$ does not intersect the two boundary geodesics, or it intersects but only on edges that are close to the other side, it does not contribute to the number $|E_{ij}|$.
    
    Suppose that $f$ intersects $\gamma_i$ and the edges of $\partial f \cap \gamma_i$ are far from the other side. There are three cases.
    
    \begin{enumerate}[{Case} 1.]
        \item $\partial f\cap \gamma_j$ has no edges.
        
        Since $\gamma_i$ is a geodesic path, $|\partial f \cap \gamma_i| \leq |\Int(f)|$. So $f$ contributes at most $|\Int(f)|$ edges to $|E_{ij}|$.

        \item $\partial f\cap \gamma_j$ contains edges, but these edges are close to the other side.

        By definition, $|\partial f \cap \gamma_j|<|\Int (f)|$. Since $\gamma_i$ is a geodesic path, $|\partial f \cap \gamma_i| \leq |\Int(f)| + |\partial f \cap \gamma_j| < 2|\Int(f)|$. So $f$ contributes at most $2|\Int(f)|$ edges to $|E_{ij}|$.
        
        \item $\partial f\cap \gamma_j$ contains edges, and these edges are far from the other side.

        $|\partial f \cap \gamma_j|<|\Int f|$ and $|\partial f \cap \gamma_i|<|\Int f|$. So $f$ contributes at most $2|\Int(f)|$ edges to $|E_{ij}|$.
    \end{enumerate}
    By summing the contributions of all faces $f\in D_{ij}$, we get \[|E_{ij}|\leq \sum_{f\in D_{ij}}2|\Int(f)|.\]
    
    Recall that the number of edges on $\partial D_{ij}$ that are not on the two boundary geodesics is at most $200\ell$. Therefore, we have
    \[\sum_{f\in D_{ij}}|\Int(f)| = 2\Cancel(D_{ij})+|\partial D_{ij}\setminus (\gamma_i\cup\gamma_j)| \leq 2\Cancel(D_{ij})+200\ell .\]
    Hence the result.

\end{proof}

\begin{lem}\label{lem reducible can not be close}
    Let $(f, f')$ be a reducible pair of faces of $Y$, and let $e$ be a common edge of $f, f'$ that contributes to the reduction degree $\Red(Y)$. Then $e$ cannot be simultaneously close to the other side in the two diagrams containing $f$ and $f'$ respectively. 
\end{lem}
\begin{proof}
    Suppose that $e\in \gamma_1$, $f\in D_{12}$, $f'\in D_{13}$ and $e$ is close to the other side both in $D_{12}$ and in $D_{13}$. Since $f$ and $f'$ form a reducible pair, their images under the associated map $Y\to\Cay(X,R_\ell)$ must coincide. Because the three geodesics $\gamma_1, \gamma_2$ and $\gamma_3$ do not intersect, the image of $f\cap \gamma_2$ does not intersect the images of $f'\cap \gamma_1$ and $f'\cap\gamma_3$, which gives $|f\cap \gamma_2|<|\Int (f')|$. For the same reason, $|f'\cap \gamma_3|<|\Int (f)|$.

    However, since $e$ is close to the other side in both diagrams, by definition, $|\Int(f)| \leq |\partial f \cap \gamma_2|$ and $|\Int(f')| \leq |\partial f' \cap \gamma_3|$. This results in a contradiction when comparing $|\Int (f)|$ and $|\Int (f')|$.
\end{proof}

\begin{lem}\label{inequality reduction degree}
    \[\Red(Y)\leq 4\sum_{1\leq i<j\leq k}\Cancel(D_{ij})+\binom{k}{2}400\ell.\]
\end{lem}
\begin{proof}
    Recall that the reduction degree of $Y$ is
    \[\Red(Y) = \sum_{e\in Y^{(1)}}\sum_{r\in R}\sum_{1\leq j\leq \ell}\Big(\big|\{f\in Y \,|\, \varphi_2(f) = r, e \textup{ is the $j$-th edge of } \partial f\}\big| - 1 \Big)^+.\]

    Let $e$ be an edge of $Y$, and let $f_1,\dots,f_t$ be a set of $t\geq 2$ faces such that each pair $f_i,f_j$ is reducible, with the reduction happening at $e$. Since every sub-diagram $D_{ij}$, $1\leq i<j\leq k$, is reduced, $e$ is on one of the geodesics, and each face $f_s$, $1\leq s\leq t$, lies on a different sub-diagram.
    
    By Lemma \ref{lem reducible can not be close}, if $e$ is close to the other side in the sub-diagram containing $f_1$, then it must be far from the other side in the other $t-1$ sub-diagrams containing respectively $f_2,\dots,f_t$. This means that every time we add $t-1$ to the reduction degree $\Red(Y)$, we also add at least $t-1$ to the sum $\displaystyle{\sum_{1\leq i<j \leq k}|E_{ij}|}$. Thus, we have
    \[\Red(Y)\leq \sum_{1\leq i<j \leq k}|E_{ij}|.\]
    The result follows from Lemma \ref{lem estimation far from the other side}.
\end{proof}\quad\\

\begin{proof}[\textbf{Proof of Theorem \ref{number of parallel geodesics}}]\quad

Combining Lemma \ref{cancel of Y}, Lemma \ref{inequality reduction degree} and Inequality (\ref{inequality number of faces}), Inequality (\ref{useful reduced inequality}) becomes
\[k(k-2)(C-8000)\ell - 3\sum_{1\leq i<j\leq k}\Cancel(D_{ij}) - \binom{k}{2}400\ell \leq (d+\varepsilon)\binom{k}{2}\frac{2C+800}{1-2d-\varepsilon}\ell.\]

Recall that $|\partial D_{ij}| \leq 2C\ell+800\ell$, so by Inequality (\ref{useful diagram inequality}),

\[\Cancel(D_{ij}) \leq \frac{d+\varepsilon}{1-2d-\varepsilon}|\partial D_{ij}| \leq \frac{2C+800}{1-2d-\varepsilon}(d+\varepsilon)\ell.\]
Thus, we get
\[k(k-2)(C-8000)\ell \leq 4\binom{k}{2}\frac{d+\varepsilon}{1-2d-\varepsilon}(2C+800)\ell+ \binom{k}{2}400\ell,\]
or
\[k(k-2) - 8\binom{k}{2}\frac{d}{1-2d} \leq 8\binom{k}{2}\frac{d+\varepsilon}{1-2d-\varepsilon} \frac{C+400}{C-8000} + \binom{k}{2}\frac{400}{C-8000} - 8\binom{k}{2}\frac{d}{1-2d}.\]

Remind that for any large $C$ and any small $\varepsilon$, the above inequality holds (a.a.s. as $\ell\to\infty$) in $G_\ell(m,d)$. Note that the right hand side of the inequality converges to $0$ as $C\to\infty$ and $\varepsilon\to 0$. If the left hand side were strictly positive, we could find a small enough $\varepsilon=\varepsilon(d)$ and a large enough $C = C(d,k)$ that contradict the inequality. Therefore,
\[k(k-2) \leq 8\binom{k}{2}\frac{d}{1-2d}.\]
Since $k\geq 1$ and $1-2d>0$, we have
\[(1-2d)(k-2)\leq 4d(k-1),\]
or
\[k(1-6d)\leq 2-8d.\]
As $d<1/6$, we obtain
\[k \leq 2 + \frac{4d}{1-6d}.\]

\end{proof}

\section{Minimal stable length at density d<1/6}

Let $G = \Pres{X}{R}$ be a hyperbolic group. The \textit{translation length} of an element $u\in G$ (acting on $\Cay(X,R)$) is defined by \[[u]:=\inf_{x\in G}|ux-x|,\]
and the \textit{stable length} is \[ [u]^\infty:=\lim_{n\to\infty}\frac{|u^nx-x|}{n}.\]

According to \cite[Ch.10 \S 6]{CDP90}, such limit exists and does not depend on the choice of $x$. Recall that $[u]\geq [u]^\infty\geq [u]-16\delta$. Recall also that $[u]^\infty>0$ if and only if $u$ is loxodromic. The \textit{minimal stable length} of the group $G = \Pres{X}{R}$ is then defined by \[\inf\SetCond{[u]^\infty}{u\in G, u \textup{ is loxodromic}}.\]

\quad

We will need the following two lemmas.

\begin{lem}\label{stable length = translation length}
    If $u\in G$ fixes a geodesic $\gamma$ and $x\in \gamma$, then $[u]^\infty = [u] = |ux-x|$.
\end{lem}
\begin{proof}
    Assume that there were some $y\not\in \gamma$ such that $|uy-y|\leq |ux-x|-1$ for some $x\in \gamma$. For any integer $n\in \mathbb{N}$, we have \[|u^n-y|\leq n|uy-y|\leq n|ux-x|-n.\]
    On the other hand, by the triangular inequality, \[|u^y-y|\geq |u^nx-x|-|u^nx-u^ny|-|x-y| = n|ux-x|-2|x-y|.\]
    So for any $n\in \mathbb{N}$, $2|x-y|\geq n$, which leads to a contradiction. Hence $[u] = |ux-x|$. 

    Since $|u^nx-x| = n|ux-x|$ for any $n$, we have $[u]^\infty = [u]$.
\end{proof}

\begin{lem}[{\cite[Proposition 3.1]{Del96}}]\label{existence of parallel geodesics} Let $u\in G$ be a loxodromic element. Let $u^\infty$ and $u^{-\infty}$ be two points on the boundary of $G$ that are the limit points of $u^n$ and $u^{-n}$ as $n\to\infty$. Then there exists a finite set of parallel geodesics connecting $u^\infty$ and $u^{-\infty}$ that are permuted by $u$.
\end{lem}

\quad

\begin{thm} \label{inj rad at 1/6} If $d<1/6$, then a.a.s. the minimal stable length of $G_\ell(m,d)$ is $1$.
\end{thm}
\begin{proof}
    Recall that for any $d<1/2$, a.a.s. $G_\ell(m,d)$ is hyperbolic and torsion free, so any non-trivial element is loxodromic. At $d<1/6$, the hyperbolicity constant is $\delta\leq 10\ell$. We shall prove that for any $u\in G_\ell(m,d)$ with $u\neq 1$, we have $[u]^\infty\geq 1$. 
    
    Since $[u]^\infty$ is invariant by conjugations, we can replace $u$ within its conjugates so that $[u] = |u|$. If $[u]\geq 20\delta$, then $[u]^\infty\geq [u]-16\delta\geq 1$. Suppose in the following that $|u| = [u]\leq 20\delta\leq 200\ell$. 
    
    Consider a finite set of parallel geodesics connecting $u^\infty$ and $u^{-\infty}$ that are permuted by $u$ (Lemma \ref{existence of parallel geodesics}). Let $\gamma$ be one such geodesic, there is a positive integer $k$ such that $u^k\gamma = \gamma$. So $\gamma, u\gamma,\dots, u^{k-1}\gamma$ is a set of parallel geodesics permuted by $u$, each fixed by $u^k$. By Theorem \ref{number of parallel geodesics}, when $d<1/6$, we have $k\leq 2+\frac{4d}{1-6d}$. In the case $k=1$, we have $[u]^\infty = [u] \geq 1$ by Lemma \ref{stable length = translation length}.
    
    Suppose that $k\geq 2$. Let $x,y$ be vertices on $\gamma$ with $|x-y| = C\ell$ where $C = C(m,d)$ is a large number to be determined. Choose geodesics $[x,ux]$ and $[y,uy]$ so that their intersections with $\gamma$ and with $u\gamma$ are connected. Let $D$ be the van Kampen diagram bounded by the geodesics $[x,y]$, $[y,uy]$, $[uy,ux]$ and $[ux,x]$, with the associated map $\varphi : D\to \Cay(G_\ell(m,d))$. It is a disk diagram with eventually $4$ spurs. Denote by $\mathring D$ the disk part of $D$. We have $2C\ell-2|u|\leq |\partial \mathring D|\leq 2|u|+2C\ell$, so $|D|\geq |\partial \mathring D|/\ell\geq 2C-800$. Together with the isoperimetric inequality $\partial \mathring D\leq (1-2d-\varepsilon)|D|\ell$ with $d<1/6$ and some small enough $\varepsilon$, we get    
    \[2C-800\leq |D|\leq 4C+800.\]
    
    Let $D'$ be the copy of $D$ that is bounded by the geodesics $u^k[x,y]$, $u^k[y,uy]$, $u^k[uy,ux]$ and $u^k[ux,x]$, with the associated map $\varphi': D'\to \Cay(G_\ell(m,d))$, so that $\varphi'(D') = u^k \varphi(D)$. For any face $f\in D$, there is a corresponding face $f'\in D'$ such that $\varphi'(f') = u^k \varphi (f)$. Since $u^k$ fixes $\gamma$ and $u\gamma$, $D'$ is also bounded by $\gamma$ and $u\gamma$.

    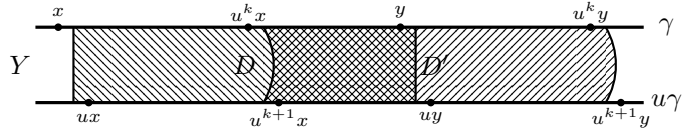
\begin{figure}[h]
    \begin{center}
    \begin{tikzpicture}
    \fill [gray!40, pattern=north east lines] (3.5,0) -- (-1,0) -- (-1,1) -- (3.5,1);
    \fill [gray!40, pattern=north east lines] (3.5,1) arc [start angle=30, end angle=-30, radius=1];
    \fill [white] (-1,1) arc [start angle=30, end angle=-30, radius=1];
    \fill [gray!40, pattern=north west lines] (-3.5,0) -- (1,0) -- (1,1) -- (-3.5,1);

    \node at (-1.25,0.5) {$D$};
    \node at (1.25,0.5) {$D'$};
    \node at (-4.2,0.5) {$Y$};

\draw [very thick] (-4,0) -- (4,0);
\draw [very thick] (-4,1) -- (4,1);
\node at (4.3,0) {$u\gamma$};
\node at (4.3,1) {$\gamma$};

\draw [thick] (-3.5,1) -- (-3.5,0);
\draw [thick] (-1,1) arc [start angle=30, end angle=-30, radius=1];
\draw [thick] (1,0) -- (1,1);
\draw [thick] (3.5,1) arc [start angle=30, end angle=-30, radius=1];

\fill (-0.8,0) circle (0.05);
\fill (1.2,0) circle (0.05);
\fill (-3.3,0) circle (0.05);
\fill (3.7,0) circle (0.05);
\fill (-1.2,1) circle (0.05);
\fill (0.8,1) circle (0.05);
\fill (-3.7,1) circle (0.05);
\fill (3.3,1) circle (0.05);
\node at (-3.7,1.2) {\scriptsize{$x$}};
\node at (3.3,1.2) {\scriptsize{$u^ky$}};
\node at (-1.2,1.2) {\scriptsize{$u^kx$}};
\node at (0.8,1.2) {\scriptsize{$y$}};
\node at (-3.3,-0.2) {\scriptsize{$ux$}};
\node at (3.7,-0.2) {\scriptsize{$u^{k+1}y$}};
\node at (-0.8,-0.2) {\scriptsize{$u^{k+1}x$}};
\node at (1.2,-0.2) {\scriptsize{$uy$}};
\end{tikzpicture}
\end{center}
    \caption{The diagram $Y$ obtained by attaching $D$ and $D'$ along $\gamma$ and $u\gamma$.}
\end{figure}

    We choose $C\geq 200k$, so that $|x-y| \geq 200k\ell\geq |u^k|$, and that $u^kx$ lie between $x$ and $y$. Let $Y$ be the diagram obtained by attaching $D$ and $D'$ by their associated maps, along the two geodesics $\gamma$ and $u\gamma$. The 2-complex $Y$ is a annular diagram with $|Y|=2|D|\geq 4C-1600$ and $|\partial Y|\leq 4|u|+4|u^k|\leq 800(k+1)\ell$. If $Y$ were reduced, then by the isoperimetric inequality, 
    \[800(k+1)\ell\geq (1-2d-\varepsilon)(4C-1600)\ell.\]
    
    We choose $C\geq\frac{1000k}{1-2d}$ so that this inequality does not hold, so $Y$ must be non-reduced. Since $D$ and $D'$ are reduced, the reductions of $Y$ must happen on the two geodesics $\gamma$ and $u\gamma$.

    Let $f\in D, g'\in D'$ be a reducible pair of faces in $Y$, we have $\varphi(f) = \varphi'(g')$ in the Cayley complex. Let $g$ be the face in $D$ that corresponds to $g'$ in $D'$, which gives $u^k\varphi(g) = \varphi'(g') = \varphi(f)$. Since $G_\ell(m,d)$ is torsion free, the faces $f$ and $g$ must be different in $D$.
    
    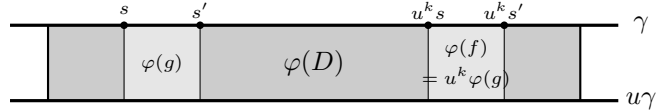
\begin{figure}[h]
\begin{center}
\begin{tikzpicture}
\fill [gray!40] (-3.5,0) -- (3.5,0) -- (3.5,1) -- (-3.5,1);
\fill [gray!20] (-1.5,0) -- (-2.5,0) -- (-2.5,1) -- (-1.5,1);
\fill [gray!20] (1.5,0) -- (2.5,0) -- (2.5,1) -- (1.5,1);
\node at (0,0.5) {$\varphi(D)$};
\node at (-2,0.5) {\scriptsize{$\varphi(g)$}};
\node at (2,0.7) {\scriptsize{$\varphi(f)$}};
\node at (2,0.3) {\scriptsize{$= u^k\varphi(g)$}};

\draw [very thick] (-4,0) -- (4,0);
\draw [very thick] (-4,1) -- (4,1);
\node at (4.3,0) {$u\gamma$};
\node at (4.3,1) {$\gamma$};

\node at (-2.5,1.2) {\scriptsize{$s$}};
\node at (-1.5,1.2) {\scriptsize{$s'$}};
\node at (1.5,1.2) {\scriptsize{$u^ks$}};
\node at (2.5,1.2) {\scriptsize{$u^ks'$}};
\fill (-2.5,1) circle (0.05);
\fill (-1.5,1) circle (0.05);
\fill (1.5,1) circle (0.05);
\fill (2.5,1) circle (0.05);

\draw [thick] (-3.5,0) -- (-3.5,1);
\draw [thick] (3.5,0) -- (3.5,1);
\draw (2.5,0) -- (2.5,1);
\draw (1.5,0) -- (1.5,1);
\draw (-2.5,0) -- (-2.5,1);
\draw (-1.5,0) -- (-1.5,1);

\end{tikzpicture}
\end{center}
    \caption{The image of the diagram $D$.}
\end{figure}

    By Lemma \ref{parallel diagram in one sixth}, as $d<1/6$, $D$ has only one layer of faces. Apply the isoperimetric inequality on the subdiagram of $D$ containing $g$ and its two neighbouring faces, we have $|\Int(g)|\leq 3(d+\varepsilon)\ell$. So either $|\varphi(g)\cap \gamma| = |g\cap \gamma|\geq(1-3d-3\varepsilon)\ell/2$, or $|g\cap u\gamma| = |\varphi(g)\cap u\gamma|\geq(1-3d-3\varepsilon)\ell/2$. Suppose that we are in the first case.
    
    Denote the segment $\varphi(g)\cap \gamma$ by $[s, s']$. Since $u^k$ fixes $\gamma$ and $u^k\varphi(g)= \varphi(f)$, we have $\varphi(f)\cap \gamma = [u^ks,u^ks']$. Because $f\neq g$, these two sub-geodesics have no common edges. So $|s-u^ks|\geq |s-s'|\geq (1-3d-3\varepsilon)\ell/2$. Hence $[u^k]\geq (1-3d-3\varepsilon)\ell/2$. In addition, as $u^k$ fixes $\gamma$, for any $i\geq 1$, $[u^{ik}] = i[u^{k}]\geq \frac{(1-3d-3\varepsilon)i}{2}\ell$. So we have $[u]^\infty \geq \frac{1-3d-3\varepsilon}{2k}\ell$, which is a lot bigger than $1$ when $\ell$ is large enough.
\end{proof}

\quad

\begin{rem}\label{rem stable length}
    According to the proof above, at $d<1/6$, any loxodromic element $u$ of $G_\ell(m,d)$ with $[u]<\frac{1-3d-3\varepsilon}{2k}\ell$ fixes a geodesic connecting $u^\infty$ and $u^{-\infty}$. Therefore, by Lemma \ref{stable length = translation length}, we have $[u]=[u]^\infty$.

    We do not have $[u]=[u]^\infty$ for long enough elements. For instance, at $d<1/4$, we can construct an element $u$ with $|u|=\frac{\ell}{2}$ and $[u]^\infty\leq (1/2-d+\varepsilon)\ell$: By the intersection formula, a.a.s. there is a relation $r$ of $G_\ell(m,d)$ written as $r = svsw$ where $|s|=(d-\varepsilon)\ell$ and $|v|=|w|=(1/2-d-\varepsilon)\ell$. Since $d<1/4$, every trivial word of $G_\ell(m,d)$ has length at least $\ell$. So $u = sv$ is a geodesic word of length $\ell/2$.

    Construct a reduced van Kampen diagram with $2$ faces, both labeled by $r$ and share a common segment of word $u$ (but at different positions). This proves that $[u^2]\leq |v|+|w|\leq (1-2d+2\varepsilon)\ell$, which gives $[u]^\infty\leq (1/2-d+\varepsilon)\ell$.
    
    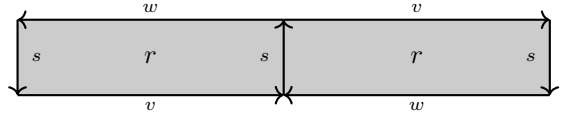
\begin{figure}[h]
\begin{center}
\begin{tikzpicture}
\fill [gray!40] (-3.5,0) -- (3.5,0) -- (3.5,1) -- (-3.5,1);
\node at (-1.75,0.5) {$r$};
\node at (1.75,0.5) {$r$};
\node at (-3.25,0.5) {\scriptsize{$s$}};
\node at (-0.25,0.5) {\scriptsize{$s$}};
\node at (3.25,0.5) {\scriptsize{$s$}};
\node at (-1.75,1.15) {\scriptsize{$w$}};
\node at (-1.75,-0.15) {\scriptsize{$v$}};
\node at (1.75,1.15) {\scriptsize{$v$}};
\node at (1.75,-0.15) {\scriptsize{$w$}};

\draw [->, thick] (-3.5,0) -- (0,0);
\draw [<-, thick] (0,0) -- (3.5,0);
\draw [<-, thick] (-3.5,1) -- (0,1);
\draw [->, thick] (0,1) -- (3.5,1);
\draw [<-, thick] (-3.5,0) -- (-3.5,1);
\draw [<-, thick] (3.5,0) -- (3.5,1);
\draw [->, thick] (0,0) -- (0,1);

\end{tikzpicture}
\end{center}
    \caption{An element $u = sv$ with $[u]^\infty < [u]$.}
\end{figure}
\end{rem}

\printbibliography[heading=bibintoc]

\end{document}